\documentclass{amsart}
\usepackage{amsfonts}
\usepackage{amsmath}
\usepackage{amssymb}
\usepackage{amsthm}

\usepackage{graphicx}
\usepackage{capt-of}
\usepackage{tikz}
\usetikzlibrary{patterns}
\usetikzlibrary{decorations.markings}
\usepackage{multirow}
\usepackage{array} 
\newcolumntype{C}{>{$}c<{$}}

\newtheorem{theorem}{Theorem}
\newtheorem{prop}[theorem]{Proposition}
\newtheorem{lemma}[theorem]{Lemma}
\newtheorem{rem}[theorem]{Remark}

\begin{document}

\title{The graphs of non-degenerate linear codes}
\author{Mark Pankov}
\subjclass[2000]{05C60, 94B05, 94B27}
\keywords{linear code, Grassmann graph, adjacency preserving map}
\address{Faculty of Mathematics and Computer Science, University of Warmia and Mazury, S{\l}oneczna 54, Olsztyn, Poland}
\email{pankov@matman.uwm.edu.pl}

\maketitle

\begin{abstract}
We consider the Grassmann graph of $k$-dimensional subspaces of an $n$-dimensional vector space over the $q$-element field
and its subgraph $\Gamma(n,k)_q$ which consists of all non-degenerate linear $[n,k]_q$ codes. We assume that $1<k<n-1$.
It is well-known that every automorphism of the Grassmann graph is induced 
by a semilinear automorphism of the corresponding vector  space or a semilinear isomorphism to the dual vector space and 
the second possibility is realized only for $n=2k$.
Our result is the following:
if $q\ge 3$ or $k\ne 2$,
then every isomorphism of  $\Gamma(n,k)_{q}$ to a subgraph of the Grassmann graph
can be uniquely extended to an automorphism of the Grassmann graph;
in the case when $q=k=2$, there are subgraphs of the Grassmann graph
isomorphic to $\Gamma(n,k)_{q}$ and such that isomorphisms between these subgraphs and $\Gamma(n,k)_{q}$
cannot be extended to automorphisms of the Grassmann graph.
\end{abstract}

\section{Introduction}
Grassmann graphs are interesting for many reasons.
First of all, these are classical examples of distance regular graphs \cite{BCN}. Also, they are closely connected to one class of Tits buildings \cite{BC,Pankov1,Pasini,Shult}.
The first significant result concerning Grassmann graphs was Chow's theorem \cite{Chow} which describes their automorphisms.
A survey of results in that spirit can be found in \cite{Pankov2}.
Recently,  Chow's theorem was used to prove some Wigner-type theorems for Hilbert Grassmannians \cite{GS,Geher,Pankov3}.

The Grassmann graph formed by $k$-dimensional subspaces of  an $n$-dimensional vector space over the $q$-element field
can be considered as the graph of  linear $[n,k]_q$ codes, where two codes are connected by an edge if 
they have the maximal number of common codewords
(we assume that $1<k<n-1$, since in the cases $k=1,n-1$ the Grassmann graph is complete). 
In practice, only non-degenerate linear codes are useful (a code is non-degenerate if every coordinate functional is not vanished on this code).
For a degenerate linear $[n,k]_q$ code every generator matrix contains zero column 
which means that we have, in fact, a linear $[n-1,k]_q$ code.

The graph (the subgraph of the Grassmann graph) formed by non-degenerate linear $[n,k]_q$ codes is investigated in  \cite{KP1,KP2}.
By \cite{KP1}, the distances in this graph coincide with the distances in the Grassmann graph if and only if
$$n<(q+1)^2 +k-2;$$
maximal cliques and automorphisms are determined in \cite{KP2}.
Subgraphs of the Grassmann graph formed by different families of linear codes are considered in \cite{CGK,KPP,KP3}.

By Chow's theorem, every automorphism of the Grassmann graph is induced  by a semilinear automorphism of  
the corresponding vector space or a semilinear isomorphism to the dual vector space
and the second possibility is realized only for $n=2k$.
We show that for the case when $q\ge 3$ or $k\ge 3$ every subgraph $\Gamma$ of the Grassmann graph isomorphic to 
the graph of non-degenerate linear $[n,k]_q$ codes
can be obtained form this graph by an automorphism of the Grassmann graph;
in other words, there is a coordinate system such that $\Gamma$ is the graph of non-degenerate codes 
or $n=2k$ and $\Gamma$ is formed by all codes dual to non-degenerate codes.

In the case when $q=k=2$, we construct a counterexample showing that the above statement fails.

\iffalse
Chow's theorem \cite{Chow} provides a description of all automorphisms of the Grassmann graph;
every such automorphism is induced  by a semilinear automorphism of  ${\mathbb F}^n_q$ or a semilinear isomorphism to the dual vector space
(the second possibility is realized only if $n=2k$).
Our main result is the following: if $q\ge 3$ or $k\ge 3$, then every subgraph of the Grassmann graph isomorphic to  $\Gamma(n,k)_{q}$
can be obtained form $\Gamma(n,k)_{q}$ by an automorphism of the Grassmann graph.
For the case when $q=k=2$ we construct counterexample.
\fi

\section{Result}
Let $V$ be an $n$-dimensional vector space over a field ${\mathbb F}$ (not necessarily finite)
and let ${\mathcal G}_{k}(V)$  be the Grassmannian consisting of all $k$-dimensional subspaces of $V$.
Two $k$-dimensional subspaces of $V$ are called {\it adjacent} if their intersection is $(k-1)$-dimensional, 
or equivalently, their sum is $(k+1)$-dimensional.
The {\it Grassmann graph} $\Gamma_{k}(V)$ is the simple graph whose vertex set is ${\mathcal G}_{k}(V)$ and 
two vertices are connected by an edge in this graph if the corresponding $k$-dimensional subspaces are adjacent.
In the cases when $k=1,n-1$, any two distinct $k$-dimensional subspaces are adjacent. 
The Grassmann graph is connected.
Chow's theorem \cite{Chow} provides a description of automorphisms of $\Gamma_{k}(V)$
for $1<k<n-1$ (if $k=1,n-1$, then every bijective transformation of ${\mathcal G}_{k}(V)$ is a graph automorphism).

Recall that a transformation $l:V\to V$ is called {\it semilinear} if 
$$l(x+y)=l(x)+l(y)$$
for all $x,y\in V$ and there is a homomorphism $\sigma: {\mathbb F}\to {\mathbb F}$
such that 
$$l(ax)=\sigma(a)l(x)$$
for all $a\in {\mathbb F}$ and $x\in V$.
In the case when $l$ is bijective and $\sigma$ is a field automorphism, 
we say that $l$ is a {\it semilinear automorphism} of $V$.
Every semilinear automorphism of $V$ induces an automorphism of the graph $\Gamma_{k}(V)$.
Let $\omega$ be a non-degenerate symmetric form on $V$.
For every subspace $X\subset V$ we denote by $X^{\perp}_{\omega}$ the orthogonal complement of $X$ associated to the form $\omega$.
Two subspaces $X,Y\subset V$ of the same dimension are adjacent if and only if their orthogonal complements $X^{\perp}_{\omega},Y^{\perp}_{\omega}$ are adjacent,
i.e. the orthocomplementary map $X\to X^{\perp}_{\omega}$ provides an isomorphism between $\Gamma_{k}(V)$ and $\Gamma_{n-k}(V)$;
this is an automorphism of $\Gamma_{k}(V)$ if $n=2k$.
Chow's theorem states that every automorphism of $\Gamma_{k}(V)$, $1<k<n-1$ is induced by a semilinear automorphism of $V$
or it is the composition of the orthocomplementary map and a graph automorphism  induced by a semilinear automorphism of $V$;
the second possibility is realized only for $n=2k$.

Suppose that ${\mathbb F}={\mathbb F}_{q}$ is the finite field consisting of $q$ elements
and $V={\mathbb F}^n$.
The standard basis of $V$ is formed by the vectors 
$$e_{1}=(1,0,\dots,0),e_2=(0,1,0,\dots,0),\dots,e_{n}=(0,\dots,0,1).$$
Let $C_{i}$ be the kernel of the $i$-th coordinate functional
$(x_{1},\dots,x_{n})\to x_{i}$, i.e. the hyperplane of $V$ spanned by all $e_j$ with $j\ne i$.
Every $k$-dimensional subspace of $V$ is  a {\it linear} $[n,k]_{q}$ {\it code}.
Such a code $C$ is {\it non-degenerate} if
the restriction of every coordinate functional to $C$ is non-zero, i.e. $C$ is not contained in any coordinate hyperplane $C_{i}$ \cite{TVN}.
All non-degenerate linear $[n,k]_{q}$ codes form the set 
$${\mathcal C}(n,k)_{q}=
{\mathcal G}_{k}(V)\setminus \left(\bigcup_{i=1}^{n}{\mathcal G}_{k}(C_{i}) \right).$$
Denote by $\Gamma(n,k)_{q}$ the restriction of the Grassmann graph $\Gamma_{k}(V)$ to the set ${\mathcal C}(n,k)_{q}$,
i.e.  $\Gamma(n,k)_{q}$ is the simple graph whose vertex set is ${\mathcal C}(n,k)_{q}$
and two vertices are connected by an edge if the corresponding $k$-dimensional subspaces are adjacent.
This graph is connected \cite{KP1}.
By \cite{KP2}, every automorphism of $\Gamma(n,k)_{q}$, $1<k<n-1$ is induced by a monomial semilinear automorphism of $V$
(a semilinear automorphism sending every $e_i$ to a scalar multiple of $e_j$).

Our main result is the following.

\begin{theorem}\label{main-1}
Suppose that $1<k<n-1$. If 
\begin{equation}\label{eq-th}
q\ge 3\;\mbox{ or }\;k\ge 3,
\end{equation}
then every isomorphism of  $\Gamma(n,k)_{q}$ to a subgraph of $\Gamma_{k}(V)$
can be uniquely extended to an automorphism of $\Gamma_{k}(V)$.
In the case when $q=k=2$, there are subgraphs of $\Gamma_{k}(V)$ 
isomorphic to $\Gamma(n,k)_{q}$ and such that isomorphisms between these subgraphs and $\Gamma(n,k)_{q}$
cannot be extended to automorphisms of $\Gamma_{k}(V)$.
\end{theorem}

Note that the subgraphs considered in Theorem  \ref{main-1} are not supposed to be induced subgraphs.

\begin{rem}{\rm
If $k=1,n-1$, then every injective map of ${\mathcal C}(n,k)_q$ to ${\mathcal G}_k(V)$
is an isomorphism of $\Gamma(n,k)_q$ to a subgraph of $\Gamma_k(V)$
and every bijective transformation of ${\mathcal G}_k(V)$ is an automorphism of $\Gamma_k(V)$.
In this case, each isomorphism of $\Gamma(n,k)_q$ to a subgraph of $\Gamma_k(V)$ can be extended to 
an automorphism of $\Gamma_k(V)$, but such an extension is not unique.
}\end{rem}

\begin{rem}{\rm
Let us take a basis of $V$ and consider the hyperplanes $H_1,\dots,H_n\subset V$ spanned by vectors from this basis. 
The restriction of $\Gamma_k(V)$ to the set
\begin{equation}\label{eq-set}
{\mathcal G}_{k}(V)\setminus \left(\bigcup_{i=1}^{n}{\mathcal G}_{k}(H_{i}) \right)
\end{equation}
is isomorphic to $\Gamma(n,k)_{q}$.
If $1<k<n-1$ and \eqref{eq-th} holds, then Theorem \ref{main-1} 
states that every subgraph $\Gamma$ of $\Gamma_{k}(V)$  isomorphic to $\Gamma(n,k)_{q}$
can be obtained from $\Gamma(n,k)_{q}$ by an automorphism of $\Gamma_{k}(V)$, i.e.
$\Gamma$ is the restriction of $\Gamma_{k}(V)$ to a subset of type \eqref{eq-set}
or $n=2k$ and $\Gamma$ is the restriction of $\Gamma_{k}(V)$ to the set formed by the orthogonal complements of 
elements from a subset of type \eqref{eq-set}.
}\end{rem}

\begin{rem}{\rm
Suppose that $V={\mathbb F}^n$ and ${\mathbb F}$ is an infinite field.
The direct analogue of the first part of Theorem \ref{main-1}  does not hold for this case.
This is related to the fact that $\Gamma_k(V)$ can contain proper subgraphs isomorphic to $\Gamma_k(V)$.
Consider two examples. 
If $\sigma$ is a non-surjective endomorphism of ${\mathbb F}$
(such endomorphisms exist, for example, for the field of complex numbers), 
then every $\sigma$-linear embedding of $V$ to itself sending bases of $V$ to bases (strong semilinear embedding)
induces an isomorphism of $\Gamma_k(V)$ to a proper subgraph of $\Gamma_k(V)$;
see \cite[Chapter 3]{Pankov2} for examples of more complicated embeddings of $\Gamma_k(V)$ to itself.
Suppose that ${\mathbb F}={\mathbb R}$ and take any bijection $g$ of ${\mathcal G}_k(V)$ to 
a maximal clique ${\mathcal C}$ of $\Gamma_k(V)$ 
(such bijections exist, since ${\mathcal G}_k(V)$ and ${\mathcal C}$ are of the same cardinality).
Consider the subgraph of $\Gamma_k(V)$ whose vertex set is ${\mathcal C}$
and $X,Y\in {\mathcal C}$ are connected by an edge if and only if $g^{-1}(X),g^{-1}(Y)$ are connected by an edge in $\Gamma_k(V)$.
All isometric embeddings of Grassmann graphs (embedding preserving the path distance) are known, 
they are induced by semilinear embeddings of special type \cite[Section 3.4]{Pankov2}.
A classification of non-isometric embeddings is an open problem \cite[Chapter 3]{Pankov2}.
}\end{rem}

\section{Preliminaries}

\subsection{Maximal cliques in the Grassmann graph}
Recall that a subset in the vertex set of a graph is called a {\it clique} if any two vertices from this subset are connected by an edge.
A clique ${\mathcal X}$ is said to be {\it maximal} if every clique containing ${\mathcal X}$ coincides with ${\mathcal X}$.

If $1<k<n-1$, then every maximal clique of $\Gamma_{k}(V)$ is of one of the following types:
\begin{enumerate}
\item[$\bullet$] the star ${\mathcal S}(X)$, $X\in {\mathcal G}_{k-1}(V)$ which consists of
all $k$-dimensional subspaces containing $X$;
\item[$\bullet$] the top ${\mathcal T}(Y)={\mathcal G}_{k}(Y)$, $Y\in {\mathcal G}_{k+1}(V)$.
\end{enumerate}
An $m$-dimensional vector space $W$ over ${\mathbb F}$ contains precisely
$$[m]_q=\frac{q^m-1}{q-1}=q^{m-1}+\dots+q+1$$
$1$-dimensional subspaces; this is also the number of hyperplanes in $W$. 
Since ${\mathcal S}(X)$, $X\in {\mathcal G}_{k-1}(V)$ can be identified with ${\mathcal G}_1(V/X)$,
every star of $\Gamma_k(V)$ consists of
$$[n-k+1]_q$$
elements.
Every top of $\Gamma_k(V)$ contains precisely
$$[k+1]_q$$
elements. 
Stars and tops have the same number of elements if and only if $n=2k$.

The intersection of two distinct maximal cliques of the same type contains at most one element.
This intersection is non-empty if and only if the associated $(k-1)$-dimensional or $(k+1)$-dimensional subspaces
are adjacent.

For a $(k-1)$-dimensional subspace  $X\subset V$ and a $(k+1)$-dimensional subspace $Y\subset V$
the intersection of the associated star and top of $\Gamma_{k}(V)$  is non-empty if and only if $X\subset Y$.
In this case, ${\mathcal S}(X)\cap {\mathcal T}(Y)$ contains precisely $q+1$ elements. 
Every such intersection is called a {\it line}  of ${\mathcal G}_{k}(V)$.

\subsection{Maximal cliques in the graph of non-degenerate linear codes}
As in the previous subsection, we assume that $1<k<n-1$.
For a $(k-1)$-dimensional subspace  $X\subset V$ and a $(k+1)$-dimensional subspace $Y\subset V$
we denote by  ${\mathcal S}^c(X)$ and ${\mathcal T}^c(Y)$ the intersection of ${\mathcal C}(n,k)_q$
with the star ${\mathcal S}(X)$ and the top ${\mathcal T}(Y)$, respectively.
Each of these intersections is a clique of $\Gamma(n,k)_q$ (if it is non-empty), but such a clique need not to be maximal.

Denote by $c(X)$ the number of the coordinate hyperplanes containing  a subspace $X\subset V$.
If $X\in{\mathcal C}(n,k-1)_q$, then
${\mathcal S}^c(X)={\mathcal S}(X)$ is a maximal clique of $\Gamma_k(V)$ and, consequently, it is a maximal clique of $\Gamma(n,k)_q$.
If $X$ is a $(k-1)$-dimensional subspace which does not belong to ${\mathcal C}(n,k-1)_q$, then 
$$0<c(X)\le n-k+1$$
and, by \cite[Lemma 1]{KP2},
\begin{equation}\label{eq-star}
|{\mathcal S}^c(X)|=(q-1)^{c(X)-1}q^{n-k-c(X)+1}
\end{equation}
(the equality does not hold if $c(X)=0$).

\begin{prop}\label{prop-star}
For a $(k-1)$-dimensional subspace $X\subset V$ 
the following assertions are fulfilled:
\begin{enumerate}
\item[{\rm (i)}] If $q\ge 3$, then ${\mathcal S}^c(X)$ is a maximal clique of $\Gamma(n,k)_q$.
\item[{\rm (ii)}]
For $q=2$ the set ${\mathcal S}^c(X)$ is  a maximal clique of $\Gamma(n,k)_q$ if and only if 
$$c(X)\le n-k-1.$$
\item[{\rm (iii)}] If ${\mathcal S}^c(X)$ is  a maximal clique of $\Gamma(n,k)_q$, 
then there is no $(k+1)$-dimensional subspace $Y\subset V$ such that 
${\mathcal S}^c(X)={\mathcal T}^c(Y)$.
\end{enumerate}
\end{prop}

\begin{proof}
The statements (i) and (ii) are Propositions 1 and 2 from \cite{KP2}.

(iii). If $X$ belongs to ${\mathcal C}(n,k-1)_q$, then ${\mathcal S}^c(X)={\mathcal S}(X)$ is a star of $\Gamma_k(V)$,
i.e. it is not contained in a top of $\Gamma_k(V)$.
Consider the case when $X\not\in{\mathcal C}(n,k-1)_q$.

Suppose that $q\ge 3$. By \eqref{eq-star},
$$|{\mathcal S}^c(X)|\ge (q-1)^{n-k}\ge (q-1)^2=q^2-2q+1\ge 3q-2q+1=q+1.$$
If ${\mathcal S}^c(X)={\mathcal T}^c(Y)$ for a certain $(k+1)$-dimensional subspace $Y\subset V$,
then ${\mathcal S}^c(X)$ is contained in the line ${\mathcal S}(X)\cap{\mathcal T}(Y)$.
Since $X\not\in{\mathcal C}(n,k-1)_q$, there is a coordinate hyperplane $C_i$ containing $X$.
The subspace $Y\cap C_i$ is $k$-dimensional (${\mathcal T}^c(Y)$ is empty if $Y\subset C_i$);
this is an element of the line ${\mathcal S}(X)\cap{\mathcal T}(Y)$ which do not belong to ${\mathcal C}(n,k)_q$.
Hence ${\mathcal S}^c(X)={\mathcal T}^c(Y)$ contains not greater than $q$ elements, a contradiction.

If $q=2$, then $c(X)\le n-k-1$ and \eqref{eq-star} shows that 
$$|{\mathcal S}^c(X)|\ge q^2=4>3=q+1$$
which means that ${\mathcal S}^c(X)$ is not contained in a line of ${\mathcal G}_k(V)$.
\end{proof}

We say that ${\mathcal S}^c(X)$ is a {\it star} of $\Gamma(n,k)_q$ 
only in the case when it is a maximal clique of $\Gamma(n,k)_q$. 
If $X$ belongs to ${\mathcal C}(n,k-1)_q$, then ${\mathcal S}^c(X)={\mathcal S}(X)$
is said to be a {\it maximal star} of $\Gamma(n,k)_q$.

\begin{rem}\label{rem-q=2}{\rm
If $q=2$ and ${\mathcal S}^c(X)$ is not a maximal clique of $\Gamma(n,k)_q$,
then one of the following possibilities is realized: 
\begin{enumerate}
\item[$\bullet$] $c(X)=n-k$ and ${\mathcal S}^c(X)$ contains  precisely two elements; 
\item[$\bullet$] $c(X)=n-k+1$ and ${\mathcal S}^c(X)$ consists of a unique element.
\end{enumerate}
See the proof of Proposition 2 from \cite{KP2} for the details.
}\end{rem}

Let $Y$ be a $(k+1)$-dimensional subspace of $V$. 
Then ${\mathcal T}^c(Y)$ is empty if $Y$ does not belong to ${\mathcal C}(n,k+1)_q$.
Consider the case when $Y\in {\mathcal C}(n,k+1)_q$.
For every $i\in\{1,\dots,n\}$ the $k$-dimensional subspace $Y\cap C_i$ does not belong to ${\mathcal T}^c(Y)$.
Some of the subspaces
$$Y\cap C_1,\dots,Y\cap C_n$$
may be coincident and we denote by $n(Y)$ the number of distinct elements in this collection.
Then
$$k+1\le n(Y)\le n$$
(the first inequality follows from the fact that $Y^{*}$ is spanned by the restrictions of the coordinate functionals to $Y$)
and ${\mathcal T}^c(Y)$ contains precisely
$$[k+1]_q-n(Y)$$
elements.
We say that ${\mathcal T}^c(Y)$ is a {\it top} of $\Gamma(n,k)_q$  if it is a maximal clique of $\Gamma(n,k)_q$. 
In this case, there is no $(k-1)$-dimensional subspace $X\subset V$ such that  ${\mathcal T}^c(Y)={\mathcal S}^{c}(X)$
(the statement (iii) of Proposition \ref{prop-star}).

\begin{prop}\label{prop-top}
The following two conditions are equivalent:
\begin{enumerate}
\item[$\bullet$] ${\mathcal T}^c(Y)$ is a top of $\Gamma(n,k)_q$ for every $Y\in {\mathcal C}(n,k+1)_q$;
\item[$\bullet$] $[k+1]_q-(q+1)>n$.
\end{enumerate}
\end{prop}

\begin{proof}
See \cite[Corollary 1]{KP2}.
\end{proof}
Since
$$[k+1]_q-(q+1)=q^{k}+\dots+q^2\ge q^k\ge 2k,$$
the second condition from Proposition \ref{prop-top} holds, for example, if $2k>n$.

\begin{rem}{\rm
In the case when the second condition from Proposition \ref{prop-top} fails,
there are $Y\in {\mathcal C}(n,k+1)_q$ such that ${\mathcal T}^c(Y)$ is a proper subset in a certain star of $\Gamma(n,k)_q$
or empty \cite[Propositions 4 and 5]{KP2}.
}\end{rem}

\section{General lemmas}

Let $f$ be an isomorphism of $\Gamma(n,k)_{q}$ to a certain subgraph $\Gamma$ of $\Gamma_{k}(V)$
and let $1<k<n-1$.
Then $f$ is an  adjacency preserving injection of ${\mathcal C}(n,k)_q$ to ${\mathcal G}_{k}(V)$, i.e.
if $X,Y\in {\mathcal C}(n,k)_q$ are adjacent, then $f(X)$ and $f(Y)$ are adjacent.
It is clear that $f$ transfers maximal cliques of $\Gamma(n,k)_{q}$ to maximal cliques of $\Gamma$.
Every maximal clique of $\Gamma$ is a (not necessarily maximal) clique of $\Gamma_{k}(V)$, i.e. it is a subset of a star or a top of $\Gamma_{k}(V)$. 
It must be pointed out that $f$ need not to be adjacency preserving in both directions: 
$f(X)$ and $f(Y)$ may be adjacent for non-adjecent $X,Y\in {\mathcal C}(n,k)_q$
(two adjacent $k$-dimensional subspaces of $V$ are not necessarily  connected by an edge of $\Gamma$).
For this reason the intersection of a maximal clique of $\Gamma_k(V)$ with the vertex set of $\Gamma$ need not to be a clique of $\Gamma$
and we cannot state that $f$ sends distinct maximal cliques of $\Gamma(n,k)_{q}$ to subsets of distinct maximal cliques of $\Gamma_{k}(V)$.

\begin{lemma}\label{lemma-g1}
One of the following possibilities is realized:
\begin{enumerate}
\item[$\bullet$] $f$ sends every maximal star of $\Gamma(n,k)_{q}$ to a star of $\Gamma_{k}(V)$;
\item[$\bullet$] $n=2k$ and $f$ transfers all maximal stars of $\Gamma(n,k)_{q}$ to tops of  $\Gamma_{k}(V)$.
\end{enumerate}
\end{lemma}

\begin{proof}
(1). If $2k<n$, then the number of elements in a star of $\Gamma_{k}(V)$ and, consequently, in a maximal star of $\Gamma(n,k)_{q}$ is greater than
the number of elements in a top of $\Gamma_{k}(V)$. 
Therefore, $f$ transfers all maximal stars of $\Gamma(n,k)_{q}$ to stars of $\Gamma_{k}(V)$. 

(2). In the case when $n=2k$, stars and tops of $\Gamma_{k}(V)$ have the same number of elements.
Thus, $f$ transfers every maximal star of $\Gamma(n,k)_q$ to a star or a top of $\Gamma_{k}(V)$.
Suppose that $X\in {\mathcal C}(n,k-1)_q$ and $f$ sends ${\mathcal S}(X)$ to a certain top of $\Gamma_{k}(V)$. 
Note that ${\mathcal C}(n,k-1)_q$ consists of one element if $q=2$ and $n=2k=4$
(the $1$-dimensional subspace containing the vector $(1,\dots,1)$); for this case the statement is trivial. 

In the general case, for any $Y\in {\mathcal C}(n,k-1)_q$ adjacent to $X$ we take a $(k+1)$-dimensional subspace $Z$ containing $X+Y$. 
Then $Z\in {\mathcal C}(n,k+1)_q$ and ${\mathcal T}^c(Z)$ contains the lines 
$${\mathcal S}(X)\cap{\mathcal T}(Z)\;\mbox{ and }\; {\mathcal S}(Y)\cap{\mathcal T}(Z)$$
which means that ${\mathcal T}^c(Z)$ is a top of $\Gamma(n,k)_{q}$.
Let ${\mathcal X}$ be a maximal clique of $\Gamma_{k}(V)$ containing $f({\mathcal T}^c(Z))$.
It is clear that $f({\mathcal S}(X))$ and ${\mathcal X}$ are distinct.
The intersection of ${\mathcal S}(X)$ and ${\mathcal T}^c(Z)$ contains more that one element and, consequently,
$$|f({\mathcal S}(X))\cap {\mathcal X}|>1.$$ 
Since $f({\mathcal S}(X))$ is a top of $\Gamma_{k}(V)$, the latter implies that ${\mathcal X}$ is a star of $\Gamma_{k}(V)$.
The intersection of ${\mathcal S}(Y)$ and ${\mathcal T}^c(Z)$ contains more than one element, i.e.
$f({\mathcal S}(Y))$ and ${\mathcal X}$ are distinct maximal cliques of $\Gamma_{k}(V)$
whose intersection contains more than one element.
Since ${\mathcal X}$ is a star of $\Gamma_{k}(V)$, $f({\mathcal S}(Y))$ is a top of $\Gamma_{k}(V)$.

So, $f({\mathcal S}(Y))$ is a top of $\Gamma_{k}(V)$ for every $Y\in {\mathcal C}(n,k-1)_q$ adjacent to $X$.
Using the connectedness of $\Gamma(n,k-1)_{q}$, we establish that $f$ sends every maximal star of $\Gamma(n,k)_{q}$ to a top of $\Gamma_{k}(V)$.

(3). Suppose that $2k>n$. In this case, ${\mathcal T}^c(Y)$ is a top of $\Gamma(n,k)_q$ for every $Y\in {\mathcal C}(n,k+1)_q$
(Proposition \ref{prop-top}).
The number of elements in every top of $\Gamma(n,k)_{q}$ is greater than the number of elements in a star of $\Gamma_{k}(V)$.
Indeed, a top of $\Gamma(n,k)_q$ contains not less that $$[k+1]_q-n$$ elements,
a star of $\Gamma_k(V)$ contains precisely $$[n-k+1]_q$$ elements and
$$[k+1]_q-n-[n-k+1]_q=\frac{q^{k+1}-q^{n-k+1}}{q-1}-n=\frac{q^{n-k+1}(q^{2k-n}-1)}{q-1}-n=$$
$$q^{n-k+1}(q^{2k-n-1}+\dots+q+1)-n\ge q^{k}-n\ge 2k-n>0.$$
Therefore, $f$ transfers tops of $\Gamma(n,k)_{q}$ to subsets in tops of $\Gamma_{k}(V)$.

Let $X\in {\mathcal C}(n,k-1)_q$.
We take any $(k+1)$-dimensional subspace $Y\subset V$ containing $X$.
Then $Y$ belongs to ${\mathcal C}(n,k+1)_q$ and
$$|{\mathcal S}(X)\cap {\mathcal T}^c(Y)|>1.$$
Since $f({\mathcal T}^c(Y))$ is contained in a top of $\Gamma_{k}(V)$,
for $f({\mathcal S}(X))$ one of the following possibilities is realized:
\begin{enumerate}
\item[(a)] $f({\mathcal S}(X))$ is a star of $\Gamma_{k}(V)$;
\item[(b)] $f({\mathcal S}(X))$ and $f({\mathcal T}^c(Y))$ both are contained in a certain top of $\Gamma_{k}(V)$.
\end{enumerate}
We need to show that (b) is impossible. 
Suppose that $f({\mathcal S}(X))$ and $f({\mathcal T}^c(Y))$ are contained in ${\mathcal T}(Z)$
for a certain $(k+1)$-dimensional subspace $Z\subset V$.
For any $Y'\in {\mathcal C}(n,k+1)_q$ containing $X$ and distinct from $Y$ we have
 $$|{\mathcal S}(X)\cap {\mathcal T}^c(Y')|>1$$
which implies that  a top  of $\Gamma_{k}(V)$ containing $f({\mathcal T}^c(Y'))$ 
intersects ${\mathcal T}(Z)$ in a subset containing more than one element and, consequently, 
$$f({\mathcal T}^c(Y'))\subset {\mathcal T}(Z).$$
Since each of ${\mathcal T}^c(Y),{\mathcal T}^c(Y')$ contains not less than
$[k+1]_q-n$ elements and their intersection is one-element,
$$|{\mathcal T}^c(Y)\cup {\mathcal T}^c(Y')|\ge 2[k+1]_q -2n-1=[k+1]_q +(q^k+\dots+q+1) -2n-1.$$
We have $2k>n$ and $1<k<n-1$ which implies that $k\ge 3$.
Therefore, 
$$|{\mathcal T}^c(Y)\cup {\mathcal T}^c(Y')|\ge [k+1]_q +q^k+q^{k-1}+q^{k-2}+q^{k-3}-2n-1\ge$$
$$[k+1]_q +2k+2(k-1)+2(k-2)+2(k-3)-2n-1=[k+1]_q +8k-2n-13\ge$$
$$[k+1]_q +8k-2(2k-1)-13=[k+1]_q +4k-11>[k+1]_q=|{\mathcal T}(Z)|$$
which means that the case (b) is not realized. 
\end{proof}

In the case when the second possibility from Lemma \ref{lemma-g1} is realized, 
we consider the composition of the orthocomplementary map and $f$.
This is an isomorphism of $\Gamma(n,k)_q$ to a subgraph of $\Gamma_{k}(V)$
sending every maximal star of $\Gamma(n,k)_q$ to a star of $\Gamma_{k}(V)$.
If this map is extendable to an automorphism of $\Gamma_{k}(V)$,
then the same holds for $f$.

From this moment, we assume that $f$ transfers every maximal star of $\Gamma(n,k)_q$ to a star of $\Gamma_{k}(V)$.
Then $f$ induces an injective map 
$$f_{k-1}:{\mathcal C}(n,k-1)_q\to {\mathcal G}_{k-1}(V)$$
such that 
$$f({\mathcal S}(X))={\mathcal S}(f_{k-1}(X))$$
for every $X\in {\mathcal C}(n,k-1)_q$.
Denote by $\Gamma_{k-1}$ the restriction of the Grassmann graph $\Gamma_{k-1}(V)$ to the image of $f_{k-1}$. 
Since the intersection of two distinct stars 
$${\mathcal S}(X),{\mathcal S}(Y),\;\;X,Y\in {\mathcal G}_{k-1}(V)$$ 
is non-empty if and only if $X,Y$ are adjacent, 
the map $f_{k-1}$ is an isomorphism of $\Gamma(n,k-1)_q$ to $\Gamma_{k-1}$.

Now, we show that $f_{k-1}$ sends every maximal star of $\Gamma(n,k-1)_q$ to a star of $\Gamma_{k-1}(V)$ if $k\ge 3$.
By Lemma \ref{lemma-g1}, we only need to consider the case when $n=2(k-1)$.
Let us take a $k$-dimensional subspace $Y\subset V$ containing the $1$-dimensional subspaces 
$$Q_1={\mathbb F}(1,\dots,1),\; Q_2={\mathbb F}(1,0,1,\dots,1),\; Q_3={\mathbb F}(0,1,1,\dots,1).$$
For any $(k-3)$-dimensional  subspace $Z\subset Y$ satisfying
$$Z\cap(Q_1+Q_2+Q_3)=0$$
the subspaces 
$$Z+Q_1+Q_2, \;Z+Q_1+Q_3,\;Z+Q_2+Q_3$$
form a clique in $\Gamma(n,k-1)_q$ which is not contained in a star. 
Therefore, ${\mathcal T}^c(Y)$ is a top of $\Gamma(n,k-1)_q$. 
If $M\in {\mathcal T}^c(Y)$, then 
$$Y\in {\mathcal S}(M)\;\mbox{ and }\;f(Y)\in {\mathcal S}(f_{k-1}(M))$$
which implies that $f_{k-1}(M)\subset f(Y)$, i.e.
$$f_{k-1}({\mathcal T}^{c}(Y))\subset {\mathcal T}(f(Y)).$$
If a $(k-2)$-dimensional subspace $X\subset Y$ contains $Q_1$, 
then $X\in {\mathcal C}(n,k-2)_q$ and
$$|{\mathcal S}(X)\cap {\mathcal T}^c(Y)|>1;$$
consequently, 
$$|f_{k-1}({\mathcal S}(X))\cap {\mathcal T}(f(Y))|>1.$$
Since $f_{k-1}({\mathcal S}(X))$ and ${\mathcal T}(f(Y))$ are distinct maximal cliques of $\Gamma_{k-1}(V)$ 
and  the intersection of two distinct tops of $\Gamma_{k-1}(V)$ contains at most one element,
$f_{k-1}({\mathcal S}(X))$ is a star of $\Gamma_{k-1}(V)$.

So, if $k\ge 3$, then $f_{k-1}$ sends every maximal star of $\Gamma(n,k-1)_q$ to a star of $\Gamma_{k-1}(V)$, 
i.e. it induces an isomorphism $f_{k-2}$ of $\Gamma(n,k-2)_q$ to a certain subgraph $\Gamma_{k-2}$ of $\Gamma_{k-2}(V)$.
We construct recursively a sequence $$f_{k}=f, f_{k-1},\dots,f_1,$$
where 
$$f_{i}:{\mathcal C}(n,i)_q\to {\mathcal G}_{i}(V)$$
is an isomorphism of $\Gamma(n,i)_q$ to a certain subgraph $\Gamma_i$ of $\Gamma_{i}(V)$ and
$$f_i({\mathcal S}(X))={\mathcal S}(f_{i-1}(X))$$
for every $X\in {\mathcal C}(n,i-1)_q$ and $i\ge 2$.

\begin{lemma}\label{lemma-g2}
The isomorphism $f_i$, $i\in \{2,\dots,k\}$ sends every star of $\Gamma(n,i)_q$ to a subset in a star of $\Gamma_i(V)$.
\end{lemma}

\begin{proof}
We prove  the statement for $i=k$ (the general case is similar).
Let $X$ be a $(k-1)$-dimensional subspace of $V$ such that ${\mathcal S}^c(X)$ is a non-maximal  star of $\Gamma(n,k)_q$.
Then $X\not\in {\mathcal C}(n,k-1)_q$.
We take $Q\in {\mathcal C}(n,1)_q$ (note that $Q\not\subset X$) and $Y\in {\mathcal C}(n,k-1)_q$
which is the sum of $Q$ and a $(k-2)$-dimensional subspace of $X$.
There are  $X_1,X_2\in {\mathcal S}^c(X)$
such that $$Z_j=Q+X_j,\;\;\;j\in \{1,2\}$$ are distinct $(k+1)$-dimensional subspaces
(otherwise, ${\mathcal S}^c(X)$ is contained in ${\mathcal T}^c(Z)$ for a certain $(k+1)$-dimensional subspace $Z\subset V$ which is impossible).
Observe that every ${\mathcal T}^c(Z_j)$ is a top of $\Gamma(n,k)_q$. 
Indeed, ${\mathcal T}^c(Z_j)$ contains $X_j$, $Q+X$ and $Q+X'$, where $X'$ is a $(k-1)$-dimensional subspace of $X_j$
distinct from $X$; these subspaces form a clique of $\Gamma(n,k)_q$ which is not contained in a star.

Each ${\mathcal T}^c(Z_j)$ intersects the star ${\mathcal S}^c(X)$ and  the maximal star ${\mathcal S}(Y)$
in subsets containing more than one element.
Since $f({\mathcal S}(Y))$ is a star of $\Gamma_k(V)$, every $f({\mathcal T}^c(Z_j))$ is contained in 
a top of $\Gamma_{k}(V)$. 
This means that for $f({\mathcal S}^c(X))$ one of the following possibilities is realized:
\begin{enumerate}
\item[(a)] $f({\mathcal S}^c(X))$ is contained in a star of $\Gamma_k(V)$;
\item[(b)] there is a top of $\Gamma_k(V)$ containing
$f({\mathcal S}^c(X))$ and both $f({\mathcal T}^c(Z_j))$, $j\in \{1,2\}$.
\end{enumerate}
Show that  the case (b) is impossible.

Suppose that $q\ge 3$.
Then
$$|{\mathcal S}^c(X)|\ge (q-1)^{n-k}$$
by \eqref{eq-star} and
$$|{\mathcal T}^c(Z_j)|\ge [k+1]_q-n=q^k+\dots+q+1-n.$$ 
Observe that
$${\mathcal T}^c(Z_1)\cap {\mathcal T}^c(Z_2)={\mathcal S}^c(X)\cap {\mathcal T}^c(Z_1)\cap {\mathcal T}^c(Z_2)=\{X+Y\}$$
and
$$|{\mathcal S}^c(X)\cap{\mathcal T}^c(Z_j)|\le q+1.$$
By  the inclusion-exclusion principle, we have
$$
|{\mathcal S}^c(X)\cup{\mathcal T}^c(Z_1)\cup{\mathcal T}^c(Z_2)|\ge
(q-1)^{n-k}+2(q^k+\dots+q+1-n)-2(q+1)-1+1\ge$$
$$2(n-k)+2(q^k +\dots+q^2)-2n=(q^k+\dots+q+1) +(q^k+\dots+q^2) -q-1-2k>$$
$$q^k+\dots+q+1=[k+1]_q$$
which contradicts (b).

Suppose that $q=2$. 
In this case, ${\mathcal S}^c(X)$ contains not less than $4$ elements (by the statement (ii) from Proposition \ref{prop-star} and  \eqref{eq-star}).
If $j\in \{1,2\}$, then
for every $(k-1)$-dimensional subspace $N\subset X_j$
the subspace $N+Q$ belongs to ${\mathcal T}^c(Z_j)$;
this subspace belongs to ${\mathcal S}^c(X)$ only in the case when $N=X$.
Therefore, each ${\mathcal T}^c(Z_j)$ contains 
$$[k]_2-1=2^{k-1}+\dots+2$$
 elements which do not belong to ${\mathcal S}^c(X)$. 
 This means that
the union of our three maximal cliques contains not less than 
$$4+2(2^{k-1}+\dots+2)>[k+1]_2$$
elements and (b) is impossible.
\end{proof}

\section{Proof of Theorem \ref{main-1}. The case $q\ge 3$}
Suppose that $q\ge 3$.
In this case, ${\mathcal S}^c(X)$ is a star of $\Gamma(n,i)_q$ for every $(i-1)$-dimensional subspace $X\subset V$
(the statement (i) of Proposition \ref{prop-star}). 
By Lemma \ref{lemma-g2}, the isomorphism $f_i$, $i\in \{2,\dots,k\}$ transfers every star of $\Gamma(n,i)_q$ to a subset in a star of $\Gamma_i(V)$.
Since the intersection of two distinct stars of $\Gamma_{i}(V)$ contains at most one element, 
the image of every star of $\Gamma(n,i)_q$ is contained in a unique star of $\Gamma_i(V)$.
Therefore, for every $j\in \{1,\dots,k-1\}$ the isomorphism $f_j$ can be extended to a map $g_j:{\mathcal G}_{j}(V)\to {\mathcal G}_{j}(V)$ 
such that 
$$f_{j+1}({\mathcal S}^c(X))\subset {\mathcal S}(g_j(X))$$
for every $X\in {\mathcal G}_{j}(X)$.
Note that 
$$g_j({\mathcal G}_j(Y))\subset {\mathcal G}_j(f_{j+1}(Y))$$
for every $Y\in {\mathcal C}(n,j+1)_q$.

\begin{lemma}\label{lemma3-1}
The map $g_1$ is bijective.
\end{lemma}

\begin{proof}
It is sufficient to show that $g_1$ is injective.
Since $q\ge 3$, we can take $P_1,P_2,P_3\in {\mathcal C}(n,1)_q$
whose sum is $3$-dimensional. 
For example, the $1$-dimensional subspaces 
$${\mathbb F}(1,\dots,1),\;{\mathbb F}(\alpha,1,\dots,1),\; {\mathbb F}(\alpha,\alpha,1,\dots,1),\;\;\alpha\ne 0,1$$
are as required. If 
$$g_1(P)=g_1(Q)=P'$$ 
for distinct $1$-dimensional subspaces $P,Q\subset V$,
then $f_2$ sends ${\mathcal S}^c(P)$ and ${\mathcal S}^c(Q)$ to subsets of ${\mathcal S}(P')$.
There is $i\in \{1,2,3\}$ such that $P_i\not\subset P+Q$. 
The maximal star  ${\mathcal S}(P_i)$ intersects ${\mathcal S}^c(P)$ and ${\mathcal S}^c(Q)$ in distinct elements
($P_i+P$ and $P_i+Q$, respectively).
Then $f_2({\mathcal S}(P_i))$ intersects ${\mathcal S}(P')$ in two distinct elements.
Since $f_2({\mathcal S}(P_i))$ is a star of $\Gamma_2(V)$, we have
$$f_2({\mathcal S}(P_i))={\mathcal S}(P')$$ which contradicts  the injectivity of $f_2$
(since $f_2({\mathcal S}^c(P))$ and $f_2({\mathcal S}^c(Q))$ both are contained in ${\mathcal S}(P')$).
\end{proof}

\begin{lemma}\label{lemma3-2}
For every $2$-dimensional subspace  $X\not\in {\mathcal C}(n,2)_q$ 
there are $X_1,X_2\in {\mathcal C}(n,1)_q$ such that 
$X+X_1$ and $X+X_2$ are distinct elements of ${\mathcal C}(n,3)_q$ and 
each of them defines a top of $\Gamma(n,2)_q$.
\end{lemma}

\begin{proof}
Without loss of generality we can assume that $X\subset C_1$, i.e. the first coordinate of every vector belonging to $X$ is zero. 
Then $X$ contains a non-zero vector $x$ such that the first and the second coordinates of $x$ are zero. 
%Indeed, if
%$$y=(0,y_2,\dots,y_n),\; z=(0,z_2,\dots,z_n)$$
%are linearly independent vectors of $X$ and $y_2\ne 0,z_2\ne 0$, then there is a linear  combination $ay+bz$
%such that the first and the second coordinates are zero.

Let $\alpha$ be a non-zero element of the field ${\mathbb F}$ distinct from $1$.
First, we consider the case when $X$ contains a $1$-dimensional subspace 
$$Q={\mathbb F} (0,0,x_3,\dots,x_n),$$
where every $x_i$ is zero or $\alpha$.
The $1$-dimensional subspaces 
$$P_{1}={\mathbb F}(1,\dots,1),\; P_{2}={\mathbb F}(\alpha,1,\dots,1),\;P_{3}={\mathbb F}(\alpha,\alpha,1,\dots,1)$$
and 
$$P'_1={\mathbb F}(1,1,1-x_{3},\dots,1-x_n),$$
$$P'_2={\mathbb F}( \alpha,1,1-x_{3},\dots,1-x_n),$$
$$P'_{3}={\mathbb F}(\alpha,\alpha,1-x_{3},\dots,1-x_n)$$
belong to ${\mathcal C}(n,1)_q$ and
$$P'_i \subset Q+P_i$$ for every $i\in \{1,2,3\}$.
Each $X+P_i$  belongs to ${\mathcal C}(n,3)_q$.
If $T$ is a $1$-dimensional subspace of $X$ distinct from $Q$,
then 
$$Q+P_i,\;T+P_i,\;T+P'_i$$
form a clique of $\Gamma(n,2)_q$ which is not contained in a star.
Therefore, ${\mathcal T}^c(X+P_i)$ is a top of $\Gamma(n,2)_q$. 

The subspace $P_1+P_2+P_3$ is $3$-dimensional. 
If $X$ is not contained in $P_1+P_2+P_3$, then we have
$$X+P_i\ne X+P_j$$
for some distinct $i,j$.
If $X$ is contained in $P_1+P_2+P_3$,
then $X$ intersects $P_2+P_3$ in a $1$-dimensional subspace $Q'$ distinct from $P_2,P_3$.
This subspace contains the vector 
$$v=(\alpha,1,\dots,1)+a(\alpha,\alpha,1,\dots,1)=(\alpha+a\alpha,1+a\alpha,1+a,\dots, 1+a)$$
for a certain non-zero $a\in {\mathbb F}$.
Since $Q'\subset X\subset C_1$, we have $\alpha+a\alpha=0$.
Then $a+1=0$ and all coordinates of $v$, except the second, are zero.
So, $X$ contains the $1$-dimensional subspace $Q'={\mathbb F}e_2$.
The $1$-dimensional subspaces
$$P_4={\mathbb F}(\alpha,\alpha,\alpha,1,\dots,1)\;\mbox{ and }\;P'_4={\mathbb F}(\alpha,\alpha-1,\alpha,1,\dots,1)$$
belong to ${\mathcal C}(n,1)_q$. Also,
$$P'_4\subset Q'+P_4$$ 
and $X+P_4\in {\mathcal C}(n,3)_q$.
If $T$ is a $1$-dimensional subspace of $X$ distinct from $Q'$,
then 
$$Q'+P_4,\;T+P_4,\;T+P'_4$$
form a clique of $\Gamma(n,2)_q$ which is not contained in a star.
Thus $X+P_4$ defines a top of $\Gamma(n,2)_q$.
Since $P_4$ is not contained in $P_1+P_2+P_3$,
$$X+P_i\ne X+P_4$$
for every $i\in\{1,2,3\}$.

In the general case, there is a monomial linear automorphism $l$ of $V$ such that $l^{-1}(X)$ contains
a non-zero vector $(0,0,x_3,\dots,x_n)$, where every $x_i$ is zero or $\alpha$.
Then all $l(P_i)$ belong to ${\mathcal C}(n,1)_q$, each $X+l(P_i)$ defines a top of $\Gamma(n,2)_q$ and
$$X+l(P_i)\ne X+l(P_j)$$
for some distinct $i,j$.
\end{proof}

Recall that the projective space associated to $V$ is the point-line geometry 
whose points are $1$-dimensional subspaces of $V$ and the lines are subsets of type ${\mathcal G}_1(X)$,
where $X$ is a $2$-dimensional subspace of $V$.
By the Fundamental Theorem of Projective Geometry, 
every bijective transformation of ${\mathcal G}_1(V)$ sending lines to lines is induced by a semilinear automorphism of $V$.

\begin{lemma}\label{lemma3-3}
The bijective transformation $g_1$ sends lines to lines.
\end{lemma}

\begin{proof}
If $X\in {\mathcal C}(n,2)_q$, then 
$$g_1({\mathcal G}_{1}(X))\subset{\mathcal G}_{1}(f_2(X)).$$
Since $g_1$ is bijective, we obtain that 
$$g_1({\mathcal G}_{1}(X))={\mathcal G}_{1}(f_2(X)).$$
Let $X$ be a $2$-dimensional subspace of $V$ which does not belong to ${\mathcal C}(n,2)_q$.
By Lemma \ref{lemma3-2}, there are $X_1,X_2\in {\mathcal C}(n,1)_q$ such that 
$$Z_1=X+X_1\;\mbox{ and }\; Z_2=X+X_2$$
are distinct elements of ${\mathcal C}(n,3)_q$ and each of them defines  a top of $\Gamma(n,2)_q$.
We have
$$|{\mathcal S}(X_i)\cap {\mathcal T}^c(Z_i)|>1$$
for $i\in\{1,2\}$ which implies that
$f_2({\mathcal T}^c(Z_i))$ is contained in a top of $\Gamma_{2}(V)$, i.e.
$$
f_2({\mathcal T}^c(Z_i))\subset {\mathcal T}(Z'_i)
$$
for some $3$-dimensional subspace $Z'_i\subset V$,  $i\in\{1,2\}$.
Every $1$-dimensional subspace $P\subset Z_i$ is contained in a certain $Y\in {\mathcal T}^c(Z_i)$
and 
$$g_1(P)\subset f_2(Y)\subset Z'_i$$
which means that $g_1({\mathcal G}_1(Z_i))$ is contained in ${\mathcal G}_{1}(Z'_i)$.
Since $g_1$ is bijective, 
$$
g_1({\mathcal G}_1(Z_i))={\mathcal G}_{1}(Z'_i)
$$
and $Z'_1\ne Z'_2$. 
We have $X=Z_1\cap Z_2$ which implies that $g_1$ sends ${\mathcal G}_1(X)$ to ${\mathcal G}_1(X')$, where $X'=Z'_1\cap Z'_2$.
\end{proof}

So, $g_1$ is induced by a semilinear automorphism $l$ of $V$. 
Then $f_2$ also is induced by $l$ and the statement is proved for $k=2$.
Consider the case when $k\ge 3$.

\begin{lemma}\label{lemma3-4}
If $X\in {\mathcal C}(n,k)_q$ contains $P\in {\mathcal C}(n,1)_q$,
then 
$$f(X)=l(X).$$
\end{lemma}

\begin{proof}
Let $Q$ be a $1$-dimensional subspace of $X$.
Consider a sequence
$$P\subset X_2\subset\dots\subset X_{k-1}\subset X_k=X,$$
where  every $X_i$ is an $i$-dimensional subspace and $X_2$ contains $Q$.
Then $X_i\in {\mathcal C}(n,i)$ for every $i\in \{2,\dots,k\}$ and 
$$f_{i-1}(X_{i-1})\subset f_i(X_i)$$
if $i\ge 3$.
This implies that 
$$g_1(Q)\subset f_{2}(X_2)\subset f(X).$$
Since $g_1$ is bijective,  we have 
$$g_1({\mathcal G}_{1}(X))={\mathcal G}_{1}(f(X)).$$
On the other hand,  $g_1$ is induced by $l$ and
$$g_1({\mathcal G}_{1}(X))={\mathcal G}_{1}(l(X))$$
which gives the claim.
\end{proof}

\begin{lemma}\label{lemma3-5}
For every $(k-1)$-dimensional subspace $X\subset V$
there are $Z_1,Z_2\in {\mathcal C}(n,k)_q$ such that $X=Z_1\cap Z_2$ and each $Z_i$ contains an element of ${\mathcal C}(n,1)_q$.
\end{lemma}

\begin{proof}
If $X$ contains an element of ${\mathcal C}(n,1)_q$, then any two distinct $k$-dimensional subspaces containing $X$ are as required.
Suppose that there is no element of ${\mathcal C}(n,1)_q$ contained in $X$. 
Let $\alpha$ be an element of ${\mathbb F}$ distinct from $0,1$.
Each of the $1$-dimen\-sional subspaces
$$P_{1}={\mathbb F}(1,\dots,1),\; P_{2}={\mathbb F}(\alpha,1,\dots,1),\;\dots,\; P_{k+1}={\mathbb F}(\,\underbrace{\alpha,\dots,\alpha}_{k},1,\dots,1)$$
is not contained in $X$, i.e. every $X+P_i$ is an element of ${\mathcal C}(n,k)_q$.
Since the subspace $P_1+\dots+P_{k+1}$ is $(k+1)$-dimensional, there are at least two distinct $X+P_i$.
\end{proof}

Let $X$ be a $(k-1)$-dimensional subspace of $V$ and let $Z_1,Z_2$ be as in Lemma \ref{lemma3-5}.
Then $f(Z_i)=l(Z_i)$ (by Lemma \ref{lemma3-4}) and
$$l(X)=l(Z_1)\cap l(Z_2)=f(Z_1)\cap f(Z_2)=g_{k-1}(X).$$
Therefore, $g_{k-1}$ is induced by $l$.
On the other hand, 
$$g_{k-1}({\mathcal G}_{k-1}(Y))\subset {\mathcal G}_{k-1}(f(Y))$$
for every $Y\in {\mathcal C}(n,k)_q$ and we obtain that $f(Y)=l(Y)$.
So, $f$ is induced by $l$.

\section{Proof of Theorem \ref{main-1}. The case $q=2$, $k\ge 3$}
Suppose that $q=2$.
In this case, ${\mathcal C}(n,1)_2$ consists of one element, the $1$-dimen\-sional subspace containing the vector $(1,\dots,1)$.
By the statement (ii) of Proposition \ref{prop-star}, for a $1$-dimensional subspace $P\subset V$ the set ${\mathcal S}^c(P)$ is a star of $\Gamma(n,2)_2$ if and only if
the weight of $P$ (the number of non-zero coordinates of the non-zero vector belonging to $P$) is greater than $2$.
Let ${\mathcal G}'_{1}(V)$ be the set of all such $1$-dimensional subspaces of $V$.
By Lemma \ref{lemma-g2}, there is a map $$g_1:{\mathcal G}'_{1}(V) \to {\mathcal G}_1(V)$$
such that 
\begin{equation}\label{eq2-1}
f_{2}({\mathcal S}^c(P))\subset {\mathcal S}(g_1(P))
\end{equation}
for every $P\in {\mathcal G}'_{1}(V)$.

\begin{lemma}\label{lemma2-1}
If $k\ge 3$, then  $g_1$ is injective.
\end{lemma}

\begin{proof}
Let $T$ be the $1$-dimensional subspace containing the vector $(1,\dots,1)$, i.e. the unique element of ${\mathcal C}(n,1)_2$. 
Then $f_2({\mathcal S}(T))$ is a star of $\Gamma_2(V)$.
If $g_1(T)=g_1(P)$ for a certain $P\in {\mathcal G}'_1(V)$ distinct from $T$,
then $f_2({\mathcal S}^c(P))$ is contained in  $f_2({\mathcal S}(T))$ which contradicts the injectivity of $f_2$.

Let $P$ and $Q$ be distinct elements of ${\mathcal G}'_{1}(V)$ different from $T$.
If $T$ is not contained in $P+Q$, then $T+P$ and $T+Q$ are distinct elements of ${\mathcal C}(n,2)_2$.
Since 
$$g_1(P)\ne g_1(T)\ne g_1(Q),$$ we have 
$$f_2(T+P)=g_1(T)+g_1(P)\;\mbox{ and }\;f_2(T+Q)=g_1(T)+g_1(Q).$$
These subspaces are distinct and, consequently, $g_1(P)\ne g_1(Q)$.

Suppose that $T\subset P+Q$ and $i$ is the weight of $P$.
First, we consider the case when $P$ contains the vector 
\begin{equation}\label{eq2-2}
(\underbrace{1,\dots,1}_{i}, 0,\dots,0).
\end{equation}
Then $Q$ contains the vector 
\begin{equation}\label{eq2-3}
(\underbrace{0,\dots,0}_{i}, 1,\dots,1).
\end{equation}
Since $P$ and $Q$ belong to ${\mathcal G}'_1(V)$, the both $i$ and $n-i$ are not less than $3$.
Let $P'$ and $Q'$ be the $1$-dimensional subspaces containing the vectors
$$(0,1,\dots,1)\;\mbox{ and }\; (1,\dots,1,0),$$
respectively.
The subspaces
$$X=P + P',\; Y=P'+Q',\;Z=Q+Q'$$
are mutually distinct elements of ${\mathcal C}(n,2)_2$, each of them does not contain $T$.
Therefore,
$$f_2(X)=g_1(P) + g_1(P'),\; f_2(Y)=g_1(P')+g_1(Q'),\;f_2(Z)=g_1(Q)+g_1(Q')$$
are mutually distinct.
The subspaces $X+Y$ and $Y+Z$ are distinct elements of ${\mathcal C}(n,3)_2$.
Note that $X+Y$ and $Y+Z$ are distinct elements of ${\mathcal C}(n,3)_2$ even if the both $i$ and $n-i$ are not less than $2$.
This fact will be used to prove Lemma \ref{lemma2-2}. We have
$$f_3(X+Y)=f_2(X)+f_{2}(Y)=g_1(P) + g_1(P')+g_1(Q'),$$
$$f_3(Y+Z)= f_2(Y)+f_2(Z)=g_1(Q) + g_1(P')+g_1(Q')$$
are distinct. The latter implies that $g_1(P)\ne g_1(Q)$.
In the general case, there is a monomial linear automorphism $l$ of $V$
transferring \eqref{eq2-2} and \eqref{eq2-3} to the non-zero vectors belonging to $P$ and $Q$, respectively. 
Applying the above arguments to the $2$-dimensional subspaces
$$P + l(P'),\; l(P')+l(Q'),\;Q+l(Q')$$
we establish that $g_1(P)\ne g_1(Q)$. 
\end{proof}

From this moment, we suppose that $k\ge 3$.

We extend $g_1$ to a transformation of ${\mathcal G}_1(V)$.
Let $P$ be a $1$-dimensional subspace of $V$ which does not belong to ${\mathcal G}'_{1}(V)$.
This is equivalent to the fact that the weight of $P$ is not greater than $2$.
In the case when $P$ is of weight $2$,
there are precisely two elements of ${\mathcal C}(n,2)_2$ containing $P$ (Remark \ref{rem-q=2}).
If $X$ and $Y$ are such elements of ${\mathcal C}(n,2)_2$, then $f_2(X),f_2(Y)$ are adjacent 
and we define $g_1(P)$ as their intersection. 
In the case when the weight of $P$ is $1$, the $2$-dimensional subspace $X$
spanned by $P$ and the vector $(1,\dots,1)$ is the unique element of ${\mathcal C}(n,2)_2$ containing $P$.
Let $P'$ and $P''$ be the $1$-dimensional subspaces of $X$ distinct from $P$.
They belong to ${\mathcal G}'_{1}(V)$ and $g_1(P')\ne g_1(P'')$.
We define $g_1(P)$ as the $1$-dimensional subspace of $f_2(X)$ distinct from $g_1(P')$ and $g_1(P'')$.

So, we obtain a transformation $g_1$ of ${\mathcal G}_1(V)$ such that \eqref{eq2-1} holds for every 
$1$-dimen\-sional subspace $P\subset V$ and
$$
g_1({\mathcal G}_1(X))\subset {\mathcal G}_1(f_2(X))
$$
for every $X\in {\mathcal C}(n,2)_2$.

\begin{lemma}\label{lemma2-2}
The extended transformation $g_1$ is bijective.
\end{lemma}

\begin{proof}
It is sufficient to show that the extension is injective.
As in the proof of Lemma \ref{lemma2-1}, we assume that 
$T$ is the $1$-dimensional subspace containing the vector $(1,\dots,1)$. 
For every $1$-dimensional subspace $P$ of weight $1$ we have $g_1(P)\ne g_1(T)$ by the definition. 
If $P$ is a $1$-dimensional subspace of weight $2$ and $g_1(P)=g_1(T)$,
then $f_2({\mathcal S}^c(P))$ is contained in  $f_2({\mathcal S}(T))$ which contradicts the injectivity of $f_2$.

Let $P$ and $Q$ be distinct $1$-dimensional subspaces of $V$ different from $T$ and such that
the weight of at least one of them is not greater than $2$.
In the case when $T\not\subset P+Q$, we use reasonings from the proof of Lemma \ref{lemma2-1}.
Suppose that $T\subset P+Q$.
If one of $P,Q$ is of weight $1$, then $g_1(P)\ne g_1(Q)$ by the definition.
For the case when the weights of $P$ and $Q$ are not less than $2$ we apply arguments used in the proof of Lemma \ref{lemma2-1}.
\end{proof}

Now, we prove the direct analogue of Lemma \ref{lemma3-3}.

\begin{lemma}\label{lemma2-3}
The bijective transformation $g_1$ sends lines to lines.
\end{lemma}

\begin{proof}
If $X\in {\mathcal C}(n,2)_2$, then $g_1({\mathcal G}_1(X))$ is contained in ${\mathcal G}_1(f_2(X))$.
Since $g_1$ is bijective, we have
$$g_1({\mathcal G}_1(X))={\mathcal G}_1(f_2(X)).$$
Consider a $2$-dimensional subspace $X\subset V$ which does not belong to ${\mathcal C}(n,2)_2$.
Let $Q_1,Q_2,Q_3$ be the $1$-dimensional subspaces of $X$.

For every $1$-dimensional subspace $P\subset V$ we denote by ${\mathcal N}(P)$ the set of all indices $i\in \{1,\dots,n\}$
such that the $i$-coordinate of the non-zero vector of $P$ is zero. The weight of $P$ is equal to $n-|{\mathcal N}(P)|$.
Since $X$ does not belong to ${\mathcal C}(n,2)_2$,  
$${\mathcal N}(Q_1)\cap {\mathcal N}(Q_2) \cap {\mathcal N}(Q_3)\ne\emptyset.$$
If $t\in \{1,2,3\}$ and $\{1,2,3\}\setminus\{t\}=\{i,j\}$, then
$${\mathcal N}(Q_t)=({\mathcal N}(Q_i)\cap {\mathcal N}(Q_j))\cup ({\mathcal N}(Q_i)\cup {\mathcal N}(Q_j))^c$$
(for $I\subset \{1,\dots,n\}$ we denote by $I^c$ the complement $\{1,\dots,n\}\setminus I$).
Therefore, if there are distinct $i,j\in\{1,2,3\}$ such that 
\begin{equation}\label{eq2-4}
{\mathcal N}(Q_i)\not\subset {\mathcal N}(Q_j),\;{\mathcal N}(Q_j)\not\subset {\mathcal N}(Q_i)\;\mbox{ and }\; {\mathcal N}(Q_i)\cup {\mathcal N}(Q_j)\ne \{1,\dots,n\},
\end{equation}
then the same holds for any distinct $i,j\in\{1,2,3\}$.
On the other hand, if for some distinct $i,j\in\{1,2,3\}$ we have 
\begin{equation}\label{eq2-5}
{\mathcal N}(Q_i)\cup {\mathcal N}(Q_j)=\{1,\dots,n\}
\end{equation}
and $\{1,2,3\}\setminus\{i,j\}=\{t\}$, then 
$${\mathcal N}(Q_t)={\mathcal N}(Q_i)\cap {\mathcal N}(Q_j).$$
Conversely, 
if ${\mathcal N}(Q_t)\subset {\mathcal N}(Q_i)$ for some distinct $i,t\in \{1,2,3\}$, then 
\eqref{eq2-5} holds for $j\in\{1,2,3\}\setminus\{i,t\}$.

So, one of the following two possibilities is realized:
\begin{enumerate}
\item[(i)] \eqref{eq2-4} holds for any distinct $i,j\in \{1,2,3\}$;
\item[(ii)] there are distinct $i,j\in \{1,2,3\}$ satisfying \eqref{eq2-5}.
\end{enumerate}

{\it The case} (i). 
The sets
$$({\mathcal N}(Q_1)\cup {\mathcal N}(Q_2))^c,\;({\mathcal N}(Q_1)\cup {\mathcal N}(Q_3))^c,\;({\mathcal N}(Q_2)\cup {\mathcal N}(Q_3))^c$$
are non-empty.  Furthermore,  these sets are mutually disjoint 
(otherwise, there is $s\in \{1,\dots,n\}$ such that the $s$-coordinate of the non-zero vector of $Q_i$  is non-zero for every $i\in\{1,2,3\}$ which is impossible).
Without loss of generality we can assume that 
$$3\not\in {\mathcal N}(Q_1)\cup {\mathcal N}(Q_2),\;2\not\in {\mathcal N}(Q_1)\cup {\mathcal N}(Q_3),\;1\not\in {\mathcal N}(Q_2)\cup {\mathcal N}(Q_3).$$
Then 
$$1,2\not\in{\mathcal N}(Q_3),\;1,3\not\in{\mathcal N}(Q_2),\;2,3\not\in{\mathcal N}(Q_1)$$
which implies that $i\in{\mathcal N}(Q_i)$ for every $i\in\{1,2,3\}$
(otherwise, there is $i\in\{1,2,3\}$ such that the $i$-coordinates of the non-zero vectors of $Q_1,Q_2,Q_3$ are non-zero which is impossible).
Let $P,P_1,P_2,P_3$ be the $1$-dimensional subspaces of $V$ containing the vectors
$$(1,\dots,1),\;(0,1,\dots,1),\;(1,0,1,\dots,1),\;(1,1,0,1,\dots,1),$$
respectively. 
Then 
$$Z=P+X\;\mbox{ and }\;Z_i=P_i+X,\;i\in\{1,2,3\}$$
belong to ${\mathcal C}(n,3)_2$; furthermore, the following assertions are fulfilled: 
\begin{enumerate}
\item[$\bullet$] $P+Q_i$ belongs to ${\mathcal C}(n,2)_2$ for every $i\in\{1,2,3\}$;
\item[$\bullet$] $P_j+Q_i$ belongs to ${\mathcal C}(n,2)_2$ if  and only if $i\ne j$.
\end{enumerate}
This implies that
\begin{equation}\label{eq2-6}
g_1(Q_i)\subset f_2(P+Q_i)\subset f_3(Z),\;\;\;i\in\{1,2,3\}
\end{equation}
and
\begin{equation}\label{eq2-7}
g_1(Q_i)\subset f_2(P_j+Q_i)\subset f_3(Z_j),\;\;\;i\ne j.
\end{equation}

It is clear that $Z\ne Z_i$ for every $i\in \{1,2,3\}$, 
\iffalse
Suppose, for example, that $Z=Z_1$. Then $P\subset Z_1$. 
Recall that $Z_1=P_1+X$. It is cleat that $P\ne P_1$ and $P\not\subset X$.
Also, $P_1+Q_1$ does not belong to ${\mathcal C}(n,2)_2$.
Therefore, $P$ is contained in $P_1+Q_2$ or $P_1+Q_3$. 
If $P\subset P_1+Q_2$, then $Q_2$ contains the vector $e_1$.
Since $1\not\in {\mathcal N}(Q_3)$, we have
$${\mathcal N}(Q_3)\subset \{2,\dots,n\}={\mathcal N}(Q_2)$$
which contradicts our assumption. Similarly, we show that $P\not\subset P_1+Q_3$ and get a contradiction. 
\fi
but we cannot state that $Z_i\ne Z_j$ for distinct $i,j$.
The subset of ${\mathcal C}(n,3)_2$ formed by $Z$ and all $Z_i$, $i\in\{1,2,3\}$ contains at least $2$ elements and
it is contained in a star of $\Gamma(n,3)_2$.
Since  $f_3$ sends stars of $\Gamma(n,3)_2$ to subsets in stars of $\Gamma_3(V)$,
the intersection of $f_3(Z)$ and all $f_3(Z_i)$ is a $2$-dimensional subspace $X'$.
By \eqref{eq2-6} and \eqref{eq2-7}, every $g_1(Q_i)$ is contained in $f_3(Z)$ and $f_3(Z_j)$ if $j\ne i$
which means that $g_1(Q_i)\subset X'$.
So, $g_1$ sends ${\mathcal G}_1(X)$ to ${\mathcal G}_1(X')$.

{\it The case} (ii). 
Without loss of generality we assume that
\begin{equation}\label{eq2-8}
{\mathcal N}(Q_1)\cup {\mathcal N}(Q_2)=\{1,\dots,n\}.
\end{equation}
As above, $P$ is the $1$-dimensional subspace containing the vector $(1,\dots,1)$ and $Z=P+X$.
Let $Q'_i$, $i\in \{1,2,3\}$ be the $1$-dimensional subspace of $P+Q_i$ distinct from $P$ and $Q_i$.
Then 
$${\mathcal N}(Q'_i)= {\mathcal N}(Q_i)^c$$
and 
$${\mathcal N}(Q'_1)\cap {\mathcal N}(Q'_2)=({\mathcal N}(Q_1)\cup {\mathcal N}(Q_2))^c=\emptyset.$$
The latter implies that $Q'_1+Q'_2$ belongs to ${\mathcal C}(n,2)_2$. Note that $Q_3$ is the third $1$-dimensional subspace of $Q'_1+Q'_2$.

By \eqref{eq2-8}, we have
$${\mathcal N}(Q_3)={\mathcal N}(Q_1)\cap {\mathcal N}(Q_2)$$
and, consequently, 
$${\mathcal N}(Q'_3)={\mathcal N}(Q'_1)\cup {\mathcal N}(Q'_2).$$
Then
$${\mathcal N}(Q_1)\cap {\mathcal N}(Q'_2)\cap {\mathcal N}(Q'_3)={\mathcal N}(Q_1)\cap {\mathcal N}(Q'_2)={\mathcal N}(Q_1)\cap {\mathcal N}(Q_2)^c\ne \emptyset$$
and 
$${\mathcal N}(Q_2)\cap {\mathcal N}(Q'_1)\cap {\mathcal N}(Q'_3)={\mathcal N}(Q_2)\cap {\mathcal N}(Q'_1)={\mathcal N}(Q_2)\cap {\mathcal N}(Q_1)^c\ne \emptyset$$
which means that $Q_1+Q'_2$ and $Q'_1+Q_2$ do not belong to ${\mathcal C}(n,2)_2$.
Note that the intersection of these $2$-dimensional subspaces is $Q'_3$.

So, $Z$ contains precisely four elements of ${\mathcal C}(n,2)_2$; these are
$$P+Q_{i},\;i\in\{1,2,3\}\;\mbox{ and }\;Q'_1+Q'_2.$$
The isomorphism $f_2$ transfers them to some $2$-dimensional subspaces of $f_3(Z)$.
Observe that
$g_1(P)$ is contained in all $f_2(P+Q_i)$ and $f_2(Q'_1+Q'_2)$ intersects 
$$f_2(P+Q_1),\;f_2(P+Q_2),\;f_2(P+Q_3)$$ 
in $g(Q'_1),g(Q'_2),g(Q_3)$ (respectively), see Fig.1.
This implies that all $g_1(Q_i)$ form a line of the projective space associated to $V$. 
\end{proof}
\begin{center}
\begin{tikzpicture}[scale=0.5]
\draw[line width=1pt] (0,0) circle (2cm);

\draw[fill=black] (90:4cm) circle (3pt);
\draw[fill=black] (330:4cm) circle (3pt);
\draw[fill=black] (210:4cm) circle (3pt);

\draw[fill=black] (270:2cm) circle (3pt);
\draw[fill=black] (150:2cm) circle (3pt);
\draw[fill=black] (30:2cm) circle (3pt);

\draw[thick,line width=1pt] (330:4cm) -- (90:4cm) -- (210:4cm);

\draw[thick,line width=1pt] (90:4cm) -- (270:2cm);

\node at (0,4.5cm) {$g_1(P)$};

\node at (-3.5cm,-2.5cm) {$g_1(Q_1)$};
\node at (3.5cm,-2.5cm) {$g_1(Q_2)$};
\node at (0,-2.5cm) {$g_1(Q_3)$};

\node at (-2.9,1.3cm) {$g_1(Q'_1)$};
\node at (3,1.3cm) {$g_1(Q'_2)$};
\end{tikzpicture}
\captionof{figure}{ }
\end{center}

So, $g_1$ is induced by a semilinear automorphism $l$ of $V$.
Since 
$$g_1({\mathcal G}_1(X))={\mathcal G}_{1}(f_2(X))$$
for every $X\in {\mathcal C}(n,2)_2$, $f_2$ also is induced by $l$.

\begin{lemma}\label{lemma2-4}
If $X\in {\mathcal C}(n,k)_2$ contains $$(1,\dots,1)\;\mbox{ or }\;v_i=(1,\dots,1)+e_i,$$
then 
$$f(X)=l(X).$$
\end{lemma}

\begin{proof}
If $X$ contains $(1,\dots,1)$, then we use arguments from the proof of Lemma \ref{lemma3-4}.

Suppose that $X$ contains $v_i$.
Since $X$ belongs to ${\mathcal C}(n,k)_2$,
the subspace $X\cap C_i$ is a hyperplane of $X$. 
For every vector $x\in X\setminus C_i$ 
the $2$-dimensional subspace $Y$ containing $x$ and $v_i$  belongs to ${\mathcal C}(n,2)_2$.
We  consider a sequence of subspaces
$$Y=Y_2\subset\dots\subset Y_{k-1}\subset Y_k=X,$$
where $Y_i\in{\mathcal C}(n,i)_2$ for every $i\in \{2,\dots,k\}$,
and establish that $l(Y)\subset f(X)$. 
Therefore, 
$$l(X\setminus C_i)\subset f(X)$$ 
which gives the claim.
\end{proof}

Denote by ${\mathcal G}'_{k-1}(V)$ the set of all $(k-1)$-dimensional subspaces $X\subset V$
satisfying 
$$c(X)\le n-k-1.$$
The statement (ii) of Proposition \ref{prop-star} says that ${\mathcal S}^c(X)$ is a star of $\Gamma(n,k)_2$
if and only if $X$ belongs to ${\mathcal G}'_{k-1}(V)$. 
By Lemma \ref{lemma-g2}, there is a map 
$$g_{k-1}:{\mathcal G}'_{k-1}(V)\to {\mathcal G}_{k-1}(V)$$
such that 
$$f({\mathcal S}^c(X))\subset {\mathcal S}(g_{k-1}(X))$$
for every $X\in {\mathcal G}'_{k-1}(V)$.
Note that $g_{k-1}$ is an extension of $f_{k-1}$.

Let  ${\mathcal G}''_{k-1}(V)$ be the set of all $(k-1)$-dimensional subspaces $X\subset V$ such that
$$c(X)\le n-k.$$
If a $(k-1)$-dimensional subspace $X\subset V$ does not belong to ${\mathcal G}''_{k-1}(V)$,
then $c(X)=n-k+1$ which implies that $X$ is spanned by some $e_{i_1},\dots,e_{i_{k-1}}$.

We extend $g_{k-1}$ to a map of ${\mathcal G}''_{k-1}(V)$.
By Remark \ref{rem-q=2}, for a $(k-1)$-dimen\-sional subspace $X\subset V$
satisfying $c(X)=n-k$ there are precisely two elements of ${\mathcal C}(n,k)_2$ containing $X$.
If $Y$ and $Y'$ are such elements of ${\mathcal C}(n,k)_2$, then $f(Y),f(Y')$ are adjacent and we define $g_{k-1}(X)$ as their intersection.

\begin{lemma}\label{lemma2-5}
For every $X\in{\mathcal G}''_{k-1}(V)$ there are $Y,Y'\in {\mathcal C}(n,k)_2$ such that 
$X=Y\cap Y'$ and each of $Y,Y'$ contains $(1,\dots,1)$ or $v_i$ for some $i$ {\rm(}$v_i$ is as in Lemma \ref{lemma2-4}{\rm)}.
\end{lemma}

\begin{proof}
If $X$ contains $(1,\dots,1)$ or $v_i$, then any two distinct $k$-dimensional subspaces containing $X$ are as required.

Suppose that $(1,\dots,1)$ and all $v_i$ do not belong to $X$.
Let $I$ be the set of all indices $i\in \{1,\dots,n\}$ such that $X$ contains a vector whose $i$-coordinate is non-zero.
Since $c(X)\le n-k$, this set contains at least $k$ indices. 
For every $i\in I$ we denote by $Y_i$ the subspace containing $X$ and $v_i$.
Let also $Y$ be the subspace containing $X$ and $(1,\dots,1)$.
By our assumption, $Y$ and all $Y_i$ are $k$-dimensional; furthermore, these subspaces belong to ${\mathcal C}(n,k)_2$.

We assert that there is  $i\in I$ such that $Y_i\ne Y$.
Indeed, if $Y=Y_i$ for all $i\in I$, then $Y$ contains the vector $(1,\dots,1)+v_i=e_i$ for every $i\in I$.
Since $|I|\ge k$,  the subspace $Y$ is spanned by some $e_{i_1},\dots,e_{i_k}$ which contradicts the fact that $(1,\dots,1)$ belongs to $Y$.
\end{proof}

Let $X\in{\mathcal G}''_{k-1}(V)$ and let $Y,Y'\in {\mathcal C}(n,k)_2$ be as in Lemma \ref{lemma2-5}.
Then 
$$f(Y)=l(Y),\;f(Y')=l(Y')$$ 
(by Lemma \ref{lemma2-4}) and
$$g_{k-1}(X)=f(Y)\cap f(Y')=l(Y)\cap l(Y')=l(Y\cap Y')=l(X).$$
Therefore, $g_{k-1}$ is induced by $l$.

Every $X\in{\mathcal C}(n,k)_2$ contains two distinct $Y,Y'\in {\mathcal G}''_{k-1}(V)$
(otherwise, there are distinct $(k-1)$-dimensional subspaces $Z,Z'\subset X$ satisfying $c(Z)=c(Z')=n-k+1$;
then $X=Z+Z'$ and $X$ is spanned by some $e_{i_1},\dots, e_{i_k}$ which is impossible).
Since 
$$g_{k-1}(Y)=l(Y)\;\mbox{ and }\;g_{k-1}(Y')=l(Y')$$ 
are distinct $(k-1)$-dimensional subspaces of $f(X)$, we have
$$f(X)=l(Y)+l(Y')=l(Y+Y')=l(X).$$
So, $f$ is induced by $l$.

\section{Proof of Theorem \ref{main-1}. The case $q=2$, $k=2$}

As in the previous section, we suppose that $q\ge 3$. 
For a non-empty subset $I\subset \{1,\dots,n\}$ we denote by $P_I$ 
the $1$-dimensional subspace of $V$ containing the vector whose $i$-coordinate is $1$ for every $i\in I$ 
and $0$ if $i\not\in I$. 
Let $H$ be the hyperplane of $V$ spanned by all $P_I$, where $I$ is an $(n-1)$-element subset of $\{1,\dots,n\}$ containing $n$.
It is easy to see that $P_{\{1,\dots,n\}}$ is not contained in $H$.
We decompose ${\mathcal C}(n,2)_2$ in three subsets ${\mathcal A}, {\mathcal B}, {\mathcal C}$ as follows:
\begin{enumerate}
\item[$\bullet$] ${\mathcal A}$ consists of all elements of ${\mathcal C}(n,2)_2$ containing $P_{\{1,\dots,n\}}$;
\item[$\bullet$] ${\mathcal B}$ is formed by all elements of ${\mathcal C}(n,2)_2$ contained in $H$;
\item[$\bullet$] ${\mathcal C}={\mathcal C}(n,2)_2\setminus({\mathcal A}\cup {\mathcal B})$.
\end{enumerate}

Consider $X\in {\mathcal C}$. This subspace intersects $H$ in a $1$-dimensional subspace
and we suppose that $P_I,P_J$  are the remaining two $1$-dimensional subspaces of $X$.
Then 
\begin{equation}\label{eq3-1}
I\cup J=\{1,\dots,n\}
\end{equation}
and $P_{I^c\cup J^c}$ is the $1$-dimensional subspace of $X$ contained in $H$
(recall that for every $S\subset\{1,\dots,n\}$ we denote by $S^c$ the complement of $S$ in $\{1,\dots,n\}$).
Then $P_{I^c},P_{J^c}$ are contained in $H$. 
Indeed, the $2$-dimensional subspaces 
$$P_{I}+P_{I^c},\;P_{J}+P_{J^c}$$
both contain $P_{\{1,\dots,n\}}$ and 
each of the $1$-dimensional subspaces $P_{\{1,\dots,n\}}$, $P_I$, $P_J$ is not in $H$.
Denote by $X^c$ the $2$-dimensional subspace of $V$ containing $P_{I^c}$ and $P_{J^c}$.
It is clear that $X^c\subset H$.
By \eqref{eq3-1}, we have $I^c\cap J^c=\emptyset$
and the third $1$-dimensional subspace of $X^c$  is $P_{I^c\cup J^c}$.
Since $P_{\{1,\dots,n\}}$ is not contained in $H$ and, consequently, in $X^c$, we have
$$I^c\cup J^c\ne \{1,\dots,n\}$$
which shows that $X^c$ does not belong to ${\mathcal C}(n,2)_2$.
Observe that $P_{I^c\cup J^c}$ is the intersection of $X$ and $X^c$.

If $Y\in {\mathcal C}$ is distinct from $X$, then $X^c$ and $Y^c$ are distinct.
Indeed, if $X^c=Y^c$ for a certain $Y\in {\mathcal C}$ distinct from $X$,
then $Y$ contains $P_{I\cap J}$ and one of the subspaces $P_I,P_J$ which contradicts the fact that $Y\in {\mathcal C}(n,2)_2$.

Let $h$ be the map of ${\mathcal C}(n,2)_2$ to ${\mathcal G}_2(V)$
which leaves fixed  every element of ${\mathcal A}\cup {\mathcal B}$ and sends every $X\in {\mathcal C}$ to $X^c$.
Since $X^c$ does not belong to ${\mathcal C}(n,2)_2$ for each $X\in {\mathcal C}$ and $X^c\ne Y^c$ for distinct $X,Y\in {\mathcal C}$, this map is injective. 

\begin{lemma}
The map $h$ is adjacency preserving. 
\end{lemma}

\begin{proof}
Let $X$ and $Y$ be adjacent elements of ${\mathcal C}(n,2)_2$.
We need to show that $h(X),h(Y)$ are adjacent.
Since the restriction of $h$ to ${\mathcal A}\cup {\mathcal B}$ is identity, 
the statement holds if $X,Y$ both belong to ${\mathcal A}\cup {\mathcal B}$.

We will use the following properties of  $Z\in {\mathcal C}$: 
\begin{enumerate}
\item[$\bullet$] the $1$-dimensional subspace $Z\cap H$ is contained in $Z^c$;
\item[$\bullet$] if $P_{I}\subset Z$ is not contained in $H$, then $P_{I^c}\subset Z^c$.
\end{enumerate}

Suppose that $X,Y\in {\mathcal C}$. If the $1$-dimensional subspace $X\cap Y$ is contained in $H$, 
then 
$$X\cap Y= X^c\cap Y^c$$
which means that $h(X)=X^c,h(Y)=Y^c$ are adjacent. 
If $X\cap Y=P_I$ is not contained in $H$, then 
$$X^c\cap Y^c=P_{I^c}$$
and $h(X)=X^c, h(Y)=Y^c$ are adjacent. 

Consider the case when $X\in {\mathcal C}$ and $Y\in {\mathcal A}\cup {\mathcal B}$.
If the $1$-dimensional subspace $X\cap Y$ is contained in $H$, then $X^c$ contains $X\cap Y$ and 
$h(X)=X^c,h(Y)=Y$ are adjacent. 
Assume that $X\cap Y=P_I$ is not contained in $H$.
Then $Y\in {\mathcal A}$. Since $P_{\{1,\dots,n\}}\subset Y$, the third $1$-dimensional subspace of $Y$ is $P_{I^c}$
and we have
$$X^c\cap Y=P_{I^c}$$
which implies that $h(X)=X^c, h(Y)=Y$ are adjacent. 
\end{proof}

Consider the restriction of $\Gamma_2(V)$ to the image of $h$. 
Removing from this restriction all edges connecting $h(X)$ and $h(Y)$ if $X,Y\in {\mathcal C}(n,2)_2$ are not adjacent,
we obtain a subgraph of $\Gamma_2(V)$ isomorphic to $\Gamma(n,2)_2$.

Now, we show that $h$ is adjacency preserving only in one direction. This means that $h$ cannot be extended to an automorphism of $\Gamma_2(V)$.

The $2$-dimensional subspaces
$$X=P_{\{1,3,\dots,n\}}+P_{\{2,\dots,n\}}\;\mbox{ and }\;Y=P_{\{1,2,4,\dots,n\}}+P_{\{3,\dots,n\}}$$
belong to ${\mathcal C}(n,2)_2$.
The third $1$-dimensional subspaces of $X$ and $Y$ are $P_{\{1,2\}}$ and $P_{\{1,2,3\}}$, respectively.
Therefore, $X$ and $Y$ are not adjacent.  
We have $X\in {\mathcal B}$ (since $H$ contains $P_{\{1,3,\dots,n\}}$ and $P_{\{2,\dots,n\}}$).
Observe that
$$P_{\{1,\dots,n\}}\subset P_{\{1,2\}}+P_{\{3,\dots,n\}},\;\; P_{\{1,\dots,n\}}\not\subset H,\;\;P_{\{1,2\}}\subset H,$$
i.e. $H$ does not contain $P_{\{3,\dots,n\}}$ and, consequently, $Y\in {\mathcal C}$.
The subspaces
$$h(X)=X=P_{\{1,3,\dots,n\}}+P_{\{1,2\}}\;\mbox{ and }\;h(Y)=Y^c=P_{\{1,2,4,\dots,n\}}+P_{\{1,2\}}$$
are adjacent.

\end{document}